\documentclass[11pt,a4paper]{amsart}
\usepackage{geometry}
\usepackage{amsmath}
\usepackage{amssymb}
\usepackage{amsthm}
\usepackage{color}
\usepackage{xspace}
\usepackage{ulem}
\usepackage{comment}
\usepackage[driverfallback=dvipdfm]{hyperref}
\newtheorem{theorem}{Theorem}
\newtheorem{lemma}[theorem]{Lemma}
\newtheorem{corollary}[theorem]{Corollary}
\newtheorem{proposition}[theorem]{Proposition}

\def\blem{\begin{lemma}}
\def\elem{\end{lemma}}
\def\bmat{\begin{pmatrix}}
\def\emat{\end{pmatrix}}

\def\ep{\varepsilon}
\def\R{\mathbb{R}}
\def\N{\mathbb{N}}

\def\bproof{\begin{proof}}
\def\eproof{\end{proof}}

\def\bald{\begin{aligned}}
\def\eald{\end{aligned}}

\def\PucF{\mathcal F}
\def\PucP{\mathcal M^+}
\def\PucM{\mathcal M^-}

\DeclareMathOperator\tr{tr}

\def\gep{\varepsilon}      
\def\ep{\gep}

\def\bcases{\begin{cases}}
\def\ecases{\end{cases}}
\def\balns{\begin{align*}}
\def\ealns{\end{align*}}
\def\bald{\begin{aligned}}
\def\eald{\end{aligned}}
\def\beq{\begin{equation}}
\def\eeq{\end{equation}}
\def\bred{\begin{color}{red}} 
\def\ered{\end{color}}
\def\gO{\Omega} 

\def\1{\mathbf{1}}

\theoremstyle{definition}
\newtheorem{definition}{Definition}
\theoremstyle{plain}

\theoremstyle{remark}
\newtheorem{remark}[definition]{Remark}

\allowdisplaybreaks[4]

\title[Logistic equations with sanctuary]{Fully nonlinear logistic equations with sanctuary}

\author[I. Birindelli]{Isabeau Birindelli}
\address[I. Birindelli]{Dipartimento di Matematica Guido Castelnuovo, Sapienza 
Universit\`a di Roma, Piazzale Aldo Moro 5, Roma, Italy.}
\email{isabeau@mat.uniroma1.it}

\author[G. Galise]{Giulio Galise}
\address[G. Galise]{Dipartimento di Matematica Guido Castelnuovo, Sapienza 
Universit\`a di Roma, Piazzale Aldo Moro 5, Roma, Italy.}
\email{galise@mat.uniroma1.it}

\author[F. Leoni]{Fabiana Leoni}
\address[F. Leoni]{Dipartimento di Matematica Guido Castelnuovo, Sapienza 
Universit\`a di Roma, Piazzale Aldo Moro 5, Roma, Italy.}
\email{leoni@mat.uniroma1.it}

\subjclass[2020]{35J60; 35D40; 35P30}
\keywords{Fully nonlinear elliptic equations; principal eigenvalue; sharp existence conditions; asymptotic analysis}

\begin{document}

\begin{abstract}
For the fully nonlinear stationary logistic equation  $\PucF(x,D^2u)+\mu u=k(x)u^p$ with $p>1$ and $k(x)\geq 0$, in a bounded domain with Dirichlet boundary condition, we determine, in terms of $\mu$, the existence and uniqueness or the nonexistence of a positive solution. Furthermore, we study the asymptotic behavior of the solutions when $\mu$ approaches the boundary points of the existence range. 
\end{abstract}

\maketitle
\section{Introduction}

\bigskip

Let $\Omega$ be a bounded domain of $\R^N$ and $p>1$. Let $k(x)$ be a bounded continuous function in $\Omega$ such that $0\leq k(x)\leq k_1$.
In this paper we study existence, nonexistence and asymptotic for the  following problem
 \begin{equation}\label{eq1}
 \left\{\begin{array}{cl}
 \PucF(x,D^2u)+\mu u=k(x)u^p & \mbox{in }\ \Omega\\
 u>0 
 & \mbox{in }\ \Omega \\
u=0 & \mbox{on }\ \partial\Omega,
\end{array}
\right.
 \end{equation}
 where $ \PucF$ is a uniformly elliptic operator, positively homogenous of degree 1.
 
 This equation has been extensively studied  when $\PucF(x,D^2u):=\Delta u$, beginning with the famous paper of Kazdan and Warner \cite{KW} who studied it because it is related to \lq\lq{\it the 
 problem of describing the set of scalar curvature functions associated with Riemannian metrics on a given connected, but not necessarily orientable, manifold of dimension $N\geq 3$}\rq\rq.
 The scalar curvature is the function $-k(x)$ which, in that paper, is supposed to be negative somewhere.
 
 Here, as far as the function $k\in C(\Omega)$ is concerned, we will consider two cases:  either it satisfies 
\begin{equation}\label{hypop1k}\tag{K1}
 \exists\, k_0>0\ \mbox{such that}\ k(x)\geq k_0>0 \ \, \forall x\in\Omega
 \end{equation}
or
\begin{equation}\label{hypop2k}\tag{K2}
\begin{cases}
\exists\,  \Omega_0\subset\subset\Omega\ \mbox{such that}\ k(x)=0\ \, \forall x\in \overline{\Omega_0}\, ,\\
\displaystyle\inf_{\Omega\backslash \Omega_1}k>0\ \ \forall\,  \Omega_1\subsetneqq\Omega \ \mbox{such that} \ \Omega_0\subset\subset\Omega_1.
\end{cases} 
 \end{equation}
In the case $\PucF=\Delta$, the hypothesis \eqref{hypop2k} corresponds to the equation studied by Ouyang \cite{Ou1, Ou2} and  Del Pino and Felmer in \cite{DP, DPF}. See also Du and Huang \cite{DuHu} for the study of the asymptotic behavior of solutions.

In general, the parabolic problem
$$
 \left\{\begin{array}{cl}
u_t- \PucF(x,D^2u)=\mu u-k(x)u^p & \mbox{in }\ (0,+\infty)\times\Omega\\
 u>0 
 & \mbox{in }\ (0,+\infty)\times\Omega \\
u=0 & \mbox{on }\ (0,+\infty)\times\partial\Omega\\
u(0,\cdot)=u_0 & \mbox{in }\  \overline\Omega,
\end{array}
\right.
$$
can be interpreted as the logistic equation in population dynamics. In that model $u(t,x)$ represents the density of the population, $\mu u$ is its natural growth, while
$-k(x)u^p$ is a logistic term that counters the natural growth of the population (lack of food, enemy, etc...). The set $\Omega_0$ where $k(x)=0$ is like an oasis, where the population 
has no enemies and it can follow its natural growth. The second order term $\PucF(x,D^2u)$ represents  the spatial diffusion of the species. When it is the Laplacian the diffusion is linear and homogenous; here,  we focus on  a
nonlinear diffusion case. For more details, we refer  to e.g. \cite{CanCos} or \cite{AMS}.  We consider here the stationary case, a first step in order to study the parabolic problem in a future work.

We shall see that, as in most works concerning the logistic equation, the principal eigenvalue plays a crucial role, acting as the threshold  which separates the  existence from the
nonexistence regime. 

For variational operators the principal eigenvalue is defined through the well known Rayleigh's quotient.  On the other hand, following the approach of Berestycki, Nirenberg, Varadhan \cite{BNV},   the notion of principal eigenvalue can be extended to a vast family of  fully nonlinear operators, by using its   characterization  through the validity of the Maximum Principle. More precisely, we will use the definition of principal eigenvalue given by
$$\lambda_1^+(\Omega):=\sup\{\mu\in\R\ \mbox{such that}\ \exists\,  \psi>0\ \mbox{and}\  \PucF(x,D^2\psi)+\mu \psi\leq 0 \  \mbox{in}\ \Omega \}.$$ 
The precise conditions on the class of operators that we consider and on the principal eigenvalue will be given in the next section, but they are the standard ones used  in the theory of viscosity solutions. It is important to remark that in  the whole paper,  when  considering solutions, subsolutions or supersolutions,  we always mean them in the viscosity sense.
 
 We will always suppose that the operator $\PucF(x,D^2\cdot)$ satisfies hypothesis \eqref{hypop}, \eqref{hypop2} and \eqref{hypop2'} given in subsection \ref{Hypotheses and notations}, and that both $\Omega$ and $\Omega_0$ are $C^2$  bounded domains.

Our main results are contained in  the following three theorems. 

The first one sharply relates the existence of solutions of \eqref{eq1} with the values of $\mu$ compared with  the principal eigenvalue.

\begin{theorem}\label{KW} Suppose that $\PucF$ satisfies \eqref{hypop}, \eqref{hypop2} and \eqref{hypop2'}.
\begin{enumerate}
\item If $k(x)$  satisfies \eqref{hypop1k}, then there exists a unique solution of \eqref{eq1} if and only if $\mu>\lambda_1^+(\Omega)$.
\item If $k(x)$ satisfies \eqref{hypop2k}, then there exists a unique solution of \eqref{eq1} if and only if $\lambda_1^+(\Omega)<\mu<\lambda_1^+(\Omega_0)$.
\end{enumerate}
\end{theorem}


\medskip
The next theorems are concerned with the  behavior of the solution for $\mu$ converging to the  boundary points of the existence range. We denote by $\phi_\Omega$  the positive solution of \eqref{eq1.1} such that $\left\|\phi_\Omega\right\|_\infty=1$.

\begin{theorem} \label{converg1} 
Suppose that $\PucF$ satisfies \eqref{hypop}, \eqref{hypop2} and \eqref{hypop2'}.
Assume that the function $k$ satisfies either assumption \eqref{hypop1k} or \eqref{hypop2k}, and let $u_\mu$ be the solution of \eqref{eq1}  for $\mu$ in the interval of existence of a solution.
Then, as  $ \mu\rightarrow \lambda_1^+(\Omega) $ from above, one has, uniformly in $\Omega$, 
\begin{equation}\label{tolomega} u_\mu\rightarrow 0 \end{equation}
and
\begin{equation}\label{tolomega2} \frac{u_\mu}{\|u_\mu\|_{L^\infty (\Omega)}}\rightarrow \phi_\gO\, .\end{equation}
\end{theorem}

For the next result, we also assume that the operator $\PucF$ and the function $k$ satisfy the assumption \eqref{H} given in the Section 4, which concerns the H\"older regularity with respect to the $x$-variable . This amounts  to the standard assumption ensuring the validity of $C^{1,\alpha}$ estimates for solutions to fully nonlinear uniformly elliptic equations.

\begin{theorem} \label{converg2} Suppose that   \eqref{hypop}, \eqref{hypop2}, \eqref{hypop2'} and \eqref{H} hold true.
 Under assumption \eqref{hypop2k}, if $u_\mu$ is the solution of \eqref{eq1}  for $\lambda_1^+(\Omega)<\mu< \lambda_1^+(\Omega_0)$, then, as  $\mu\rightarrow \lambda_1^+(\Omega_0)$ from below, one has
$$u_\mu\rightarrow +\infty \ \hbox{uniformly in}\ \overline{\Omega_0}\, ,$$
$$
\frac{u_\mu}{\|u_\mu\|_{L^\infty (\Omega_0)}}\rightarrow \phi_{\Omega_0}\ \hbox{uniformly in}\ \overline{\Omega_0}
$$
and
$$u_\mu\to \underline{U}\quad \hbox{ locally uniformly in $\Omega\setminus \overline{\Omega_0}$,}$$
 where $\underline U$ is the
minimal  solution of
 \begin{equation}\label{eq2}
 \left\{\begin{array}{cl}
 \PucF(x,D^2u)+\lambda_1^+(\Omega_0) u=k(x)u^p & \mbox{in }\ \Omega\setminus\overline\Omega_0\\
 u>0 
 & \mbox{in }\ \Omega \setminus\overline\Omega_0\\
u=0 & \mbox{on }\ \partial\Omega\\
u=+\infty & \mbox{on }\ \partial\Omega_0.
\end{array}
\right.
 \end{equation}
\end{theorem}

We wish to emphasize that the approach taken here, based on the comparison principle and the construction of barriers,  provides new proofs even for the classical case 
$\PucF=\Delta$ and naturally explains the role of the principal eigenvalue in the existence results.  
Of particular interest  is the construction of a supersolution in Lemma \ref{supso}.

In the next section we state the precise conditions on the operator, we recall the definition and main properties of the principal eigenvalue and we provide some key tools such as the maximum principle and the comparison principle.
In Section 3 we prove Theorem \ref{KW}  and the existence of the minimal and maximal blow-up solutions of problem \eqref{eq2}.
Finally, in Section 4 we prove the asymptotic results.

\section{Comparison principles and preliminaries}
\subsection{Hypotheses and notations}\label{Hypotheses and notations}
In the whole paper the operator $\PucF\in C(\Omega\times\mathcal{S}_N)$ will be uniformly elliptic, i.e. 
 $\exists \Lambda\geq\lambda>0$ such that
 \begin{equation}\label{hypop}\tag{F1}
\lambda  \tr P  \leq \PucF(x,X+P)- \PucF(x,X)\leq \Lambda\tr P
\end{equation}
 for all  $ x\in\Omega$ and any $X,P\in \mathcal{S}_N$,  $P\geq 0 $, $\mathcal{S}_N$ being the set of $N$-squared symmetric matrices.
 
 Furthermore we will suppose:
 \begin{equation}\label{hypop2}\tag{F2}
\PucF(x,tX)=t\PucF(x,X)\qquad \forall t\geq0,\,X\in\mathcal{S}_N
\end{equation}
and  there exists a modulus of continuity $\omega$ such that 
 \begin{equation}\label{hypop2'}\tag{F3}
\PucF(x,X)-\PucF(y,Y)\leq\omega(\alpha|x-y|^2+|x-y|)
\end{equation}
whenever $\alpha>0$, $x,y\in\Omega$, $X,Y\in\mathcal{S}_N$ and 
\begin{equation*}
-3\alpha\left(\begin{array}{cc}
I & 0\\
0 &  I
\end{array}\right)\leq
\left(\begin{array}{cc}
X & 0\\
0 &  -Y
\end{array}\right)\leq3\alpha\left(\begin{array}{cc}
I & -I\\
-I &  I
\end{array}\right).
\end{equation*}
Clearly, this   condition is trivially satisfied  whenever the operator does not depend  on the variable $x$.

In the following we will often use the distance function. Let us fix some notations.
We will use
$d_\Omega(x)$ to indicate the distance function to the boundary of $\Omega$ and, when no ambiguities arise, we will just write $d(x)$. For $\delta>0$, let us set $\Omega_\delta:=\{x\in\overline \Omega,\;:\, d_\Omega(x)< \delta\}$. Recall that if the $\Omega$ is $C^2$ so is $d_\Omega$ in $\Omega_\delta$ for $\delta$ sufficiently small and that $|\nabla d|=1$.

\medskip

Another key role is played by the extremal Pucci operators $\PucP$ and $\PucM$. Precisely, for $0< \lambda\leq  \Lambda$ given,  the Pucci's supremum operator is defined as
\begin{equation}\label{M+}
 \mathcal {M}^+ ( M) =\mathcal{M}^+_{\lambda, \Lambda} (M)= \Lambda \sum_{e_i\geq0}e_i+\lambda\, \sum_{e_i<0}e_i\, ,
 \end{equation}
where $e_1,\ldots,e_N$ are the eigenvalues of the matrix $M$ belonging to the space $\mathcal{S}_N$ of the $N$-squared symmetric matrices. Equivalently, $ \mathcal {M}^+$ can be defined as
$$ \mathcal {M}^+ ( M)=  \sup_{A\in {\mathcal A}_{\lambda,\Lambda}} {\rm tr} (AM)$$
where ${\mathcal A}_{\lambda,\Lambda}=\left\{A\in{\mathcal S}_N\,:\,\lambda \, I_N\leq A\leq \Lambda \, I_N\right\}$, $I_N$ being the identity matrix.

Symmetrically,  Pucci's infimum operator is defined as
 \begin{equation}\label{M-} 
 \mathcal {M}^- ( M) =\mathcal{M}^-_{\lambda, \Lambda} (M)= \lambda \sum_{e_i\geq0}\mu_i+\Lambda\, \sum_{e_i<0}e_i
=\inf_{A\in {\mathcal A}_{\lambda,\Lambda}} {\rm tr} (AM)\,.
\end{equation}
\subsection{Principal eigenvalues}
As mentioned in the introduction, we will use the principal eigenvalue \`a la Berestycki, Nirenberg Varadhan i.e.
$$\lambda_1^+(\Omega):=\sup\{\mu\in\R\ \mbox{such that}\ \exists \psi>0\ \mbox{and}\  \PucF(x,D^2\psi)+\mu \psi\leq 0 \  \mbox{in}\ \Omega \}.$$
Of course the functions $\psi$ considered in the definition of $\lambda_1^+$ are meant to be supersolutions in the viscosity sense.

The reason $\lambda_1^+$ is called an \lq\lq  eigenvalue\rq\rq\ is that 
it is possible to prove that there exists a function $\phi_\gO$, positive  in $\Omega$, called eigenfunction solution of 
 \begin{equation}\label{eq1.1}
 \left\{\begin{array}{cl}
 \PucF(x,D^2\phi_\gO)+\lambda_1^+(\Omega) \phi_\gO=0 & \mbox{in }\ \Omega\\
\phi_\gO=0 & \mbox{on }\ \partial\Omega,
\end{array}
\right.
 \end{equation}
see \cite{BD, QS}. The homogeneity condition \eqref{hypop2} is crucial for the existence of the eigenfunction $\phi_\gO$ and it implies that, without loss of generality, 
we can and will always choose $\phi_\gO$ such that $\|\phi_\gO\|_\infty=1$. 

We here recall the results concerning the relationship between the maximum principle and the \lq\lq principal eigenvalue\rq\rq.
\begin{proposition} \label{maxp} 
If $\mu<\lambda_1^+(\Omega)$ and  $u$ satisfies
$$
 \PucF(x,D^2u)+\mu u\geq 0\ \mbox{in }\ \Omega\ \mbox{and}\ 
u\leq 0 \ \mbox{on }\ \partial\Omega\, ,
$$
then $u\leq 0$.  
If $u\geq 0$ satisfies 
$$
 \PucF(x,D^2u)+ \lambda_1^+(\Omega)u\geq 0\ \mbox{in }\ \Omega\ \mbox{and}\ 
u\leq 0 \ \mbox{on }\ \partial\Omega\, ,
$$
then there exists $\tau\geq 0$ such that $u(x)=\tau\phi_\Omega(x)$
\end{proposition} 
The proof goes back to \cite{BNV} in the linear case and we refer to \cite{BD,QS} for the fully nonlinear case. Let us remark that in \cite{QS} the operator is supposed to be convex.
The second result of Proposition \ref{maxp}  is a consequence of  \cite[Theorem 4.1]{QS}. 
In that Proposition, the hypothesis that the operator be convex is not needed. Indeed, the proof holds true also for non convex operators since the convexity of the operator is used only in formula (4.1) of that paper; but in formula (4.1), the left hand side can be substituted by the maximal Pucci operator $\PucP$ and then one can use the convexity of  $\PucP$.

We now recall the following comparison principle, proved in Theorem 3.6 in \cite{BD}.
\begin{proposition} \label{compF} Let $f,g\in C(\Omega)$ be such that $f\leq 0$ and $f\leq g$ in $\Omega$. If $\mu<\lambda_1^+(\Omega)$ and $v$ and $u$ are respectively super and subsolution of
$$ \PucF(x,D^2v)+\mu v\leq f,\  \PucF(x,D^2u)+\mu u\geq g\, \mbox{in}\ \Omega$$
such that $v\geq 0$ in $\Omega$ and $v\geq u$ on $\partial \Omega$, then $v\geq u$ in $\Omega$ if 
\begin{itemize}
\item either $f<0$ in $\Omega$ 
\item or  $g(\bar x)>0$ for any $\bar x$ such that $f(\bar x)=0$.
\end{itemize}
\end{proposition}

A consequence of Proposition \ref{compF} is
\begin{corollary}\label{compmu}
Let $v$ and $u$ be respectively super and subsolution of
$$ \PucF(x,D^2v)+\mu v\leq 0,\  \PucF(x,D^2u)+\mu u\geq 0\, \mbox{in}\ \Omega.$$
Suppose that  $\mu>0$ and $\mu<\lambda_1^+(\Omega)$.
If $v>0$ in $\Omega$ and $v\geq u>0$ on $\partial\Omega$ then $v\geq u$ in $\Omega$.
\end{corollary}
\begin{proof} 
Consider $\varepsilon>0$ small enough and $\gamma_\ep$ such that on $\partial\Omega$, $v_\ep=v-\ep>u_{\gamma_\ep}=\frac{u}{1+\gamma_\ep}$.
On the other hand, $u_\ep$ and $v_{\gamma_\ep}$ satisfy respectively
$$ \PucF(x,D^2v_\ep)+\mu v_\ep\leq -\mu\ep<0\quad\,\text{in $\Omega$}$$
and
$$ \PucF(x,D^2u_{\gamma_\ep})+\mu u_{\gamma_\ep}\geq 0\quad\,\text{in $\Omega$}.$$
 Hence we are in the hypothesis of Proposition \ref{compF} and
$v_\ep\geq u_{\gamma_\ep}$ in $\Omega$. Letting $\ep\rightarrow 0$ we get the required result.
\end{proof}

\subsection{Comparison principles}
The main goal of this subsection is to prove the following comparison principle.
\begin{proposition} \label{comp}
Let $u_1>0$ and $u_2\geq 0$ in $\Omega$ be satisfying 
$$ \PucF(x,D^2u_1)+\mu u_1\leq k(x)u_1^p, \ \    \PucF(x,D^2u_2)+\mu u_2\geq k(x)u_2^p, \ \mbox{in }	 \Omega$$
with $k(x)\geq 0$ not identically zero. 

If $u_1\geq u_2$ on $\partial\Omega$ then $u_1\geq u_2$ in $\Omega$.
\end{proposition}

In order to do so we will use a few classical results that will be useful throughout the paper.
The first concerns difference of sub and supersolutions (see e.g. \cite{CC} and \cite{GV}) for which we will recall the proof in our setting,
 the other is the famous  Alexandrov-Bakelman-Pucci estimate for fully non linear operators see e.g. \cite{CC}, whose proof is not recalled here, and its consequence i.e. the maximum principle in small domain. The last ingredient for the proof of Proposition \ref{comp} is a strong comparison principle which is certainly expected and somehow known but, for completeness sake and convenience of the reader, we will give its proof.

\begin{proposition}\label{GV} Let $u\in USC(\Omega)$ and $v\in LSC(\Omega)$ be respectively viscosity subsolution and supersolution  of
$$\PucF(x,D^2u)=f(x,u)\quad\;\text{in $\Omega$}$$
and
$$\PucF(x,D^2v)=g(x,v)\quad\;\text{in $\Omega$, }$$
where $f,g\in C(\Omega\times\mathbb R)$  and $\PucF$ satisfies \eqref{hypop}, \eqref{hypop2'}. Then the difference $w = u- v$ is a
viscosity solution of the maximal differential inequality
$$\PucP (D^2w)\geq f(x,u)-g(x,v)\quad\;\text{in $\Omega$.}$$
\end{proposition}
\begin{proof}
Let $x_0\in\Omega$ and let $\varphi\in C^2(\Omega)$ be such that 
$$(u-v)(x_0)=\varphi(x_0)\;\quad\text{and}\;\quad (u-v)(x)<\varphi(x)\quad\forall x\in\Omega\backslash\left\{x_0\right\}.$$ 
Fix $r>0$ sufficiently small such that $\overline {B_{r}(x_0)}\subset\Omega$. For $\alpha>0$, let $x_\alpha,y_\alpha\in \overline {B_{r}(x_0)}$ be such that 
\begin{equation}\label{8geneq1}
\max_{ \overline {B_{r}(x_0)}\times\overline {B_{r}(x_0)}}\left((u-\varphi)(x)-v(y)-\frac\alpha2|x-y|^2\right)=(u-\varphi)(x_\alpha)-v(y_\alpha)-\frac\alpha2|x_\alpha-y_\alpha|^2.
\end{equation}
In view of \cite[Lemma 3.1]{CIL} one has  $\displaystyle\lim_{\alpha\to+\infty}\alpha|x_\alpha-y_\alpha|^2=0$ and,  using the strict inequality in \eqref{8geneq1},  we may further assume, up to extracting a subsequence,  that $\displaystyle\lim_{\alpha\to+\infty}x_\alpha=\displaystyle\lim_{\alpha\to+\infty}y_\alpha=x_0$. Moreover 
\begin{equation*}
\begin{split}
v(x_0)\leq\liminf_{\alpha\to+\infty}v(y_\alpha)&\leq\limsup_{\alpha\to+\infty}v(y_\alpha)\\&\leq\limsup_{\alpha\to+\infty}(u-\varphi)(x_\alpha)\leq (u-\varphi)(x_0)=v(x_0)\\
v(x_0)\leq\liminf_{\alpha\to+\infty}v(y_\alpha)&\leq\liminf_{\alpha\to+\infty}(u-\varphi)(x_\alpha)\\
&\leq\limsup_{\alpha\to+\infty}(u-\varphi)(x_\alpha)\leq (u-\varphi)(x_0)=v(x_0).
\end{split}
\end{equation*}
The above inequalities imply that 
\begin{equation}\label{8geneq2}
\lim_{\alpha\to+\infty}v(y_\alpha)=v(x_0)\;\quad\text{and}\;\quad\lim_{\alpha\to+\infty}u(x_\alpha)=u(x_0).
\end{equation}
By \cite[Theorem 3.2]{CIL}, there exist $X_\alpha,Y_\alpha\in\mathcal{S}_N$ such that
\begin{equation*}-3\alpha\left(\begin{array}{cc}
I & 0\\
0 &  I
\end{array}\right)\leq
\left(\begin{array}{cc}
X_\alpha & 0\\
0 &  -Y_\alpha
\end{array}\right)\leq3\alpha\left(\begin{array}{cc}
I & -I\\
-I &  I
\end{array}\right)
\end{equation*}
 and 
\begin{equation*}
\left(\alpha(x_\alpha-y_\alpha)+D\varphi(x_\alpha),X_\alpha+D^2\varphi(x_\alpha)\right)\in\overline J^{2,+}u(x_\alpha)\;,\;\;\left(\alpha(x_\alpha-y_\alpha),Y_\alpha\right)\in\overline J^{2,-}v(x_\alpha).
\end{equation*}
Using the sub and supersolution properties of $u$ and $v$ and \eqref{hypop},\eqref{hypop2'}, we obtain
\begin{equation}\label{8genpeq1}
\begin{split}
f(x_\alpha,u(x_\alpha))-g(y_\alpha,v(y_\alpha))&\leq \PucF(x_\alpha,X_\alpha+D^2\varphi(x_\alpha))-\PucF(y_\alpha,Y_\alpha)\\
&\leq \PucP (D^2\varphi(x_\alpha))+\PucF(x_\alpha,X_\alpha)-\PucF(y_\alpha,Y_\alpha)\\
&\leq \PucP (D^2\varphi(x_\alpha))+\omega(\alpha|x_\alpha-y_\alpha|^2+|x_\alpha-y_\alpha|).
\end{split}
\end{equation}
By letting $\alpha\to+\infty$ and using \eqref{8geneq2}, we conclude that
$$
f(x_0,u(x_0))-g(x_0,v(x_0))\leq\PucP (D^2\varphi(x_0)).$$ 
\end{proof}

We now recall the Alexandroff-Bakelman-Pucci estimate for viscosity subsolutions, see e.g. \cite[Chapter 3]{CC}.

\begin{theorem}\label{abp}

Suppose that $u$ is a viscosity solution of 
$${\mathcal M^+}(D^2u)\geq f\ \mbox{in}\ \Omega,$$
then there exists $C=C(N,\lambda,\Lambda,\rm{diam}\,\Omega)$ such that
$$\sup_{\Omega} u\leq \sup_{\partial\Omega}u+C\|f^-\|_{L^N(\Omega)}.$$
\end{theorem}
\begin{corollary}\label{smalld} Let $D>0$. For any domain $\Omega$ such that  ${\rm{diam}}\,\Omega\leq D$, and any given $\mu\in\mathbb R$, there exists $\delta=\delta(N,\lambda,\Lambda,\mu,D)>0$ such that if $|\Omega|\leq\delta$ and $v$ is solution of 
$$\PucP(D^2v)+\mu v\geq 0\ \mbox{in}\ \Omega, \ v\leq 0\  \mbox{on}\ \partial\Omega$$
then $v\leq 0$ in $\Omega$.
\end{corollary}
This is a well-known fact, we provide the proof because it is brief and may help the forgetful reader. 
\begin{proof}  If $\mu \leq 0$ it follows from the standard maximum principle for proper operators with uniformly elliptic principal part which is independent of $x$. So we can suppose that $\mu>0$ Suppose by contradiction that $\sup_{\Omega} v^+>0$. Since 
$$
{\mathcal M^+}(D^2v)\geq -\mu v^+\ \mbox{in}\ \Omega,
$$ by Theorem \ref{abp} we get that 
$$\sup_{\Omega} v\leq C\|\mu v^+\|_{L^N(\Omega)}\leq C\mu \delta^{\frac{1}{N}}\sup_{\Omega} v^+.$$
Hence if $C\mu \delta^{\frac{1}{N}}<1$ we get a contradiction. Hence the choice of $\delta$ depends only on $N$, $\lambda$, $\Lambda$, $D$ and $\mu$.
\end{proof}

Let us now provide a strong comparison principle:
\begin{proposition} \label{strongcomp}
Let $u_1\geq0$ and $u_2\geq0$ in $\Omega$ satisfying 
$$ \PucF(x,D^2u_1)+\mu u_1\leq k(x)u_1^p, \ \    \PucF(x,D^2u_2)+\mu u_2\geq k(x)u_2^p, \ \mbox{in }	 \Omega$$
with $k(x)\geq 0$. 

If   $u_1\geq u_2$ in $\Omega$ then either  $u_1> u_2$ in $\Omega$ or $u_1\equiv u_2$.
\end{proposition}

\begin{proof}  By contradiction, we can find $x_0,x_1\in\Omega$ with $u_1(x_1)=u_2(x_1)$, $u_1(x_0)>u_2(x_0)$ and a ball  $B_R(x_0)\subset\Omega$ such that  
\begin{equation}\label{1201eq1}
u_1(x)>u_2(x)\quad\forall x\in\overline {B_{\frac R2}(x_0)}.
\end{equation}and $x_1\in B_R(x_0)$. We introduce $w(x)=\delta(e^{-c|x-x_0|}-e^{-cR})$ with $c$ and $\delta\in(0,1)$ positive constants to be determined in such a way  
$$\PucF(x,D^2(u_2+w))-k(x)(u_2+w)^p>\PucF(x,D^2u_1)-k(x)u^p_1\quad\,\text{in $B_{R}(x_0)\setminus B_{\frac{R}{2}}(x_0)$}$$and
$$
u_2+w\leq u_1\quad\,\text{on $\partial\left(B_{R}(x_0)\setminus B_{\frac{R}{2}}(x_0)\right)$.}
$$  
It is easy to see that for $c$ sufficiently large, e.g. such that 
$$
\lambda c^2-\frac{2\Lambda(N-1)c}{R}-\sup_ {B_{R}(x_0)\setminus B_{\frac{R}{2}}(x_0)}(k(x){(u_2(x)+1)}^{p-1})>1,
$$
we get, in $B_{R}(x_0)\setminus B_{\frac{R}{2}}(x_0)$,
\begin{equation*}
\begin{split}
\PucF(x,D^2(u_2+w))&-k(x)(u_2+w)^p\geq \PucF(x,D^2u_2)+{\mathcal M}^-(D^2w)-k(x)(u_2+w)^p\\
&\geq-\mu u_2 +{\mathcal M}^-(D^2w)+k(x)(u^p_2-(u_2+w)^p)\\
&\geq -\mu u_2+\delta e^{-c|x-x_0|}\left(\lambda c^2-\frac{2\Lambda(N-1)c}{R}-\sup_ {B_{R}(x_0)\setminus B_{\frac{R}{2}}(x_0)}(k(x){(u_2+1)}^{p-1})\right)\\&>-\mu u_2+\delta e^{-cR},
\end{split}
\end{equation*}
while
$$\PucF(x,D^2u_1)- k(x)u_1^p\leq -\mu u_1\leq-\mu u_2\quad\text{in  $\Omega$}. $$
In the above inequality, we used the assumption $\mu\geq0$. The case $\mu<0$  is more standard, and the details are left to the reader. 
Clearly $u_2+w\leq u_1$ on $\partial B_{R}(x_0)$. Moreover, using \eqref{1201eq1}, we infer that for $\delta$ sufficiently small we also have 
$u_2+w\leq u_1$ on $\partial B_{\frac R2}(x_0)$.
Since the operator $\PucF(x,D^2\cdot)- k(x)(\cdot)^p$ is proper we can apply comparison principle and we get that 
in $B_{R}(x_0)\setminus B_{\frac{R}{2}}(x_0)$
 $$u_1(x)\geq u_2(x) +w(x).$$ 
 This is a contradiction, since $x_1\in B_{R}(x_0)\setminus B_{\frac{R}{2}}(x_0)$.
\end{proof}

We are now in a position to prove the comparison principle of Proposition \ref{comp}.
\begin{proof}[Proof of Proposition \ref{comp}] The proof consists of  two steps.

\smallskip
\noindent
{\bf Step 1:} {\it There exists $t_0\in(0,1]$ such that for $t\leq t_0$, $tu_2\leq u_1$ in $\Omega$.}

\smallskip
\noindent
Let $K$ be compact subset of $\Omega$ such that $|\Omega\setminus K|\leq\delta$ for $\delta$ sufficiently small. By the positivity of $u_1$ and $u_2$ there exists $t_0$ such that $t_0 u_2\leq u_1$ in $K$. We want to prove that for $\delta$ sufficiently small $t_0 u_2\leq u_1$ in $\Omega$. 
Suppose by contradiction that there exists $\bar x\in \Omega\setminus K$ such that $t_0 u_2(\bar x)>u_1(\bar x)$. Hence, the open set
$$\Omega_K^+:=\{x\in \Omega\setminus K \ \mbox{such that}\  t_0 u_2( x)>u_1( x)\}$$
is non empty and of measure smaller then $\delta$.

By Propositions \ref{GV}, in $\Omega_K^+$, the function $w(x):=t_0 u_2(x)-u_1(x)$ is a positive solution of
$$\PucP (D^2w)+\mu w\geq k(x) (t_0u_2^p-u_1^p)\geq k(x)\left({\left(t_0u_2\right)}^p-u_1^p\right)\geq0.$$
Hence, if $\delta$ is as in Corollary \ref{smalld}, then $w\leq 0$ contradicting the definition of $\Omega_K^+$.
This ends the proof of Step 1.

\smallskip

Now let 
$$\tau=\sup\{t\leq 1\,:\; tu_2\leq u_1\ \mbox{in}\ \Omega\}.$$
Observe that $t_0\leq \tau \leq 1$ and by continuity $\tau u_2\leq u_1$ in $\Omega$.

\smallskip
\noindent{\bf Step 2:} $\tau=1$.

\smallskip
\noindent
Suppose by contradiction that $\tau<1$. By the strong comparison principle Proposition \ref{strongcomp} either $\tau u_2\equiv u_1$ or $\tau u_2< u_1$.


The first case is not possible. Indeed, by Proposition \ref{GV}, the function $w=\tau u_2-u_1$, is a viscosity solution of 
$$
\PucP (D^2w)+\mu w\geq k(x)\left(\tau^{1-p}-1\right)u_1^p\quad\;\text{in $\Omega$}. 
$$
Since $w\equiv0$, the previous inequality yields a contradiction
$$0<k(x)\left(\tau^{1-p}-1\right)u_1^p(x)\leq0\quad\,\forall x\in\Omega\cap\{x:\ k(x)>0\}.$$

If, instead $\tau u_2< u_1$, we can choose $K$ as in the Step 1, hence there exists $\varepsilon>0$ such that $(\tau+\varepsilon)u_2\leq u_1$ in $K$. Finally, reasoning as above, $(\tau+\varepsilon)u_2\leq u_1$ in $\Omega$, contradicting the definition of $\tau$.

Therefore $\tau=1$ and $u_2\leq u_1$ in $\Omega$.
\end{proof}

\begin{remark} {\rm The above proposition applies also if $\Omega$ is replaced by $\Omega\setminus\overline\Omega_0$.}
\end{remark}

Combining Propositions \ref{strongcomp}-\ref{comp} we obtain the following
\begin{corollary}\label{mon} The solution $u_\mu$  of \eqref{eq1} is monotone in $\mu$, i.e. if $\mu_1<\mu_2$ then $u_{\mu_1}<u_{\mu_2}$ in $\Omega$.
\end{corollary}
For $\phi\in C(\partial\Omega)$, let us denote by $\psi_\phi$ be the solution of
\begin{equation}
 \left\{\begin{array}{cl}
 \PucF(x,D^2\psi)=0 & \mbox{in }\ \Omega\\
 \psi=\phi & \mbox{on }\ \partial\Omega\,.
\end{array}
\right.
\end{equation}
\begin{corollary}\label{bdryest} Let $\mu\in \R$ and let $\phi\in C(\partial\Omega),\,k\in C(\Omega)$ be nonnegative functions such that $k\not\equiv0$ in $\Omega_\delta$ for any $\delta>0$. Consider the problem
\begin{equation}\label{eq1.10}
 \left\{\begin{array}{cl}
 \PucF(x,D^2u)+\mu u=k(x)u^p & \mbox{in }\ \Omega\\
 u\geq0 & \mbox{in }\ \Omega \\
u=\phi & \mbox{on }\ \partial\Omega.
\end{array}
\right.
 \end{equation}
If $u$ is a subsolution of \eqref{eq1.10}, then there exists $C=C(\lambda,\Lambda,N,\mu,\Omega,\left\|\psi_\phi\right\|_\infty,\sup_\Omega u)$ positive constant,  such that
\begin{equation}\label{20mareq2}
u(x)\leq \psi_\phi(x) +Cd(x)\qquad \forall x \in\Omega.
\end{equation}
If $u>0$ is a supersolution of \eqref{eq1.10}, then there exists $C=C(\lambda,\Lambda,N,\mu,\Omega,\left\|\psi_\phi\right\|_\infty)$ positive constant,  such that
\begin{equation}\label{20mareq2'}
u(x)\geq \sup\{\psi_\phi(x) -Cd(x),0\}\qquad \forall x \in\Omega.
\end{equation}
\end{corollary}
\begin{proof} 
Let  $\delta\in(0,1)$ be sufficiently small such that the distance function $d(x)\in C^2({\Omega_\delta})$.
Clearly it is enough to prove both inequalities \eqref{20mareq2}-\eqref{20mareq2'} in $\Omega_\delta$.

\smallskip
\noindent
Fix $\alpha\in(1,2)$ and consider the nonnegative function $v(x)=d(x)-d^\alpha(x)$. For $x\in\Omega_\delta$, the eigenvalues of 
$$
D^2v(x)=(1-\alpha d^{\alpha-1}(x))D^2d(x)-\alpha(\alpha-1)d^{\alpha-2}(x)Dd(x)\otimes Dd(x)
$$
are $e_1(D^2v(x))=-\alpha(\alpha-1)d^{\alpha-2}(x)<0$ and $e_i(D^2v(x))=(1-\alpha d^{\alpha-1}(x))e_i(D^2d(x))$, where $e_i(D^2d(x))$ are,  for $i=1,\ldots,N-1$, the eigenvalues of $D^2d(x)$ corresponding to eigenvectors orthogonal to $Dd(x)$. 

Set $\overline w=\psi_\phi+Cv$, where 
\begin{equation}\label{20mareq1}
C=\max\left\{1,\frac{\left\|u\right\|_\infty}{\delta-\delta^\alpha}\right\}.
\end{equation}
Since $\psi_\phi=\phi$ on $\partial\Omega$ and $\psi_\phi\geq0$ in $\Omega$, by \eqref{20mareq1} we infer that $\overline w\geq u$ on $\partial\Omega_\delta$. In order to prove \eqref{20mareq2}, it suffices to show that $\overline w$ is a supersolution of
\begin{equation}\label{20mareq3}
\PucF(x,D^2 \overline w)+\mu \overline w=0\,\quad\text{in $\Omega_\delta$}
\end{equation}  
and then apply Proposition \ref{comp} in $\Omega_\delta$. Note that $k\not\equiv0$ in $\Omega_\delta$ by assumption. \\
A straightforward computation yields
\begin{equation*}
\begin{split}
\PucF(x,D^2 \overline w(x))+\mu\overline w(x) &\leq C{\mathcal M}^+(D^2v(x))+\mu \overline w(x)\\
&\leq Cd^{\alpha-2}(x)\big(-\lambda\alpha(\alpha-1)+(N-1)\Lambda(1+\alpha) \left\|D^2d\right\|_\infty d^{2-\alpha}(x)\\
&\hspace{2cm}+|\mu|(\left\|\psi_\phi\right\|_\infty+1) d^{2-\alpha}(x)\big)\qquad\text{in $\Omega_\delta$}.
\end{split}
\end{equation*} 
Reducing $\delta$ if necessary,  the right-hand side of the above inequality is negative since $\alpha\in(1,2)$ and $\lambda>0$. Thus $\overline w$ is supersolution of \eqref{20mareq3} and the estimate \eqref{20mareq2} follows.
%
%
%
%
%
%

\medskip

For \eqref{20mareq2'},  we consider $\underline w(x)=\sup\{\psi_\phi-Cv,0\}$ where 
$$
C=\max\left\{1,\frac{\left\|\psi_\phi\right\|_\infty}{\delta-\delta^\alpha}\right\}.
$$
Note that, with this choice of  $C$, we have that $\underline w\leq u$ on $\partial\Omega_\delta$.  Moreover, by computations similar to the above one,  for  $\delta$ sufficiently small we have
\begin{equation*}
\begin{split}
\PucF(x,D^2(\psi_\phi-Cv)(x))&+ \mu (\psi_\phi-Cv)(x)-k(x)|\psi_\phi-Cv|^{p-1}(\psi_\phi-Cv)(x)\\  &\geq -C\PucP(D^2v(x))-|\mu|(\left\|\psi_\phi\right\|_\infty +C)-\|k\|_\infty\|\psi_\phi\|^p_\infty\\
&\geq Cd^{\alpha-2}(x)\big(\lambda\alpha(\alpha-1)-2(N-1)\Lambda(1+\alpha) \left\|D^2d\right\|_\infty d^{2-\alpha}(x)\\&\hspace{2cm}-|\mu|(\left\|\psi_\phi\right\|_\infty +1)d^{2-\alpha}(x)\big)
\geq0\qquad\text{in $\Omega_\delta$}.
\end{split}
\end{equation*} 
Since the null function is a trivial subsolution of the same inequality, we have obtained that $\underline w(x)$ is a nonnegative subsolution of
$$ \PucF(x,D^2\underline w)+\mu \underline w\geq k(x)\underline w^p\quad\,\text{in $\Omega_\delta$}.$$
 Applying Proposition \ref{comp} we get the desired result.
\end{proof}

\section{Existence and nonexistence of solutions.}

Part of Theorem \ref{KW}, concerns the nonexistence of positive solutions  for $\mu\leq\lambda_1^+(\Omega)$ and $\mu\not\in (\lambda_1^+(\Omega),\lambda_1^+(\Omega_0))$. This result follows quite immediately from the definition of  principal eigenvalues and their properties.

\begin{proposition}\label{23marprop1}
 Suppose that $\PucF$ satisfies \eqref{hypop}, \eqref{hypop2} and \eqref{hypop2'}.
\begin{enumerate}
\item If $k\in C(\Omega)$  satisfies \eqref{hypop1k}, then there are no positive solutions of \eqref{eq1} if $\mu\leq\lambda_1^+(\Omega)$.
\item If $k\in C(\Omega)$ satisfies \eqref{hypop2k}, then there are no positive  solutions of \eqref{eq1} if $\mu\not\in (\lambda_1^+(\Omega),\lambda_1^+(\Omega_0))$.
\end{enumerate}
\end{proposition}
\begin{proof}
The nonexistence of solutions of \eqref{eq1} for $\mu\leq \lambda_1^+(\Omega)$ under hypothesis \eqref{hypop1k} or \eqref{hypop2k} relies only on the fact that $k$ is nonnegative and not identically zero.

If $\mu<\lambda_1^+(\Omega)$, since $u$ satisfies
$$
 \PucF(x,D^2u)+\mu u\geq 0\ \mbox{in }\ \Omega\ \mbox{and}\ 
u\leq 0 \ \mbox{on }\ \partial\Omega,
$$
 by Proposition \ref{maxp} we infer that $u\leq 0$.  Hence for any $k(x)\geq 0$, if $\mu<\lambda_1^+(\Omega)$  there can be no positive solutions of \eqref{eq1}.
 
If $\mu=\lambda_1^+(\Omega)$,  again, using Proposition \ref{maxp} there exists $\tau$ such that  $u(x)=\tau\phi_\Omega(x)$ where $\phi_\Omega$ is the eigenfunction corresponding to $\lambda_1^+(\Omega) $. But this is a contradiction since $k(x)\not\equiv 0$.

As far as the nonexistence for $\mu\geq \lambda_1^+(\Omega_0)$, we first note that the case $\mu> \lambda_1^+(\Omega_0)$ comes from the definition of $\lambda_1^+(\Omega_0)$.
Indeed, if there exists a solution $u$ of \eqref{eq1} for $\mu> \lambda_1^+(\Omega_0)$, then in $\Omega_0$, $u>0$ and 
$$\PucF(x,D^2u)+\mu u=0.$$
But, this implies by definition that $\lambda_1^+(\Omega_0)\geq\mu$.

If $\mu=\lambda_1^+(\Omega_0)$, we suppose by contradiction that there exists $u>0$ solution of \eqref{eq1}. This implies that there exists $\delta>0$ such that $u\geq\delta>0$ in $\Omega_0$. Let $M=\max_{\overline{\Omega_0}} u$ and $\varepsilon>0$ such that $\varepsilon<\frac{\lambda_1^+(\Omega_0)\delta}{2M-\delta}$. This implies that
$v=u-\frac{\delta}{2}$ satisfies in $\Omega_0$
$$
 \PucF(x,D^2v)+(\lambda_1^+(\Omega_0) +\varepsilon)v=-\lambda_1^+(\Omega_0)\frac\delta2+\varepsilon v\leq 0.
$$
This contradicts the definition of $\lambda_1^+(\Omega_0)$ and it ends the proof of the non existence results.
\end{proof}

Next, we are  going to concentrate on
 the existence part of Theorem \ref{KW}. We begin with the proof of the existence of a supersolution.
 \begin{lemma}\label{supso} Given $\mu\in (\lambda_1^+(\gO),\lambda_1^+(\gO_0))$ there exists $w^+\in C(\overline \Omega)$ positive supersolution of
$$ \left\{\begin{array}{cl}
 \PucF(x,D^2w^+)+\mu w^+= k(x)(w^+)^p & \mbox{in }\ \Omega\\
w^+= 0 & \mbox{on }\ \partial\Omega.
\end{array}
\right.
$$
\end{lemma}
\begin{proof}
Let $\Omega_1$ and $\Omega_2$ be smooth subdomains such that $$\Omega_0\subset\subset\Omega_1\subset\subset\Omega_2\subset\subset\Omega$$ and 
\begin{equation}\label{7geneq1}
\mu<\lambda_1^+(\gO_2).
\end{equation}
Let $\phi_{\Omega_2}$ be the positive solution of 
\begin{equation*}
 \left\{\begin{array}{cl}
 \PucF(x,D^2\phi_{\gO_2})+\lambda^+_1(\Omega_2) \phi_{\gO_2}=0 & \mbox{in }\ \Omega_2\\
\phi_{\gO_2}=0 & \mbox{on }\ \partial\Omega_2\\
\left\|\phi_{\gO_2}\right\|_\infty=1 & \,
\end{array}
\right.
 \end{equation*}
 and set, for $x\in\overline{\Omega_2}$,
$$\widetilde{\phi_{\gO_2}}(x)=\frac{C}{\displaystyle\min_{\overline{\gO_1}}\,\phi_{\gO_2}}\phi_{\gO_2}(x),\qquad $$
where $C$ is a positive constant to be determined. Note that by construction
\begin{equation}\label{7geneq2}
\widetilde{\phi_{\gO_2}}(x)\geq C\qquad \forall x\in\overline{\Omega_1}
\end{equation}
and that
\begin{equation}\label{7geneq3}
\PucF(x,D^2\widetilde{\phi_{\gO_2}})+\mu\widetilde{\phi_{\gO_2}}\leq0\qquad \text{in $\Omega_2$}.
\end{equation}
Now we fix $\delta=\delta(\Omega_0,\Omega_1)$ positive and sufficiently small such that
$$
\Omega_0\subset\subset\Omega_1\setminus\overline{\Omega_{1,\delta}}
$$
and the distance function $d(x):=d_{\Omega_1}(x)$ is in $C^2(\overline{\Omega_{1,\delta}})$. Consider the function
$$
\psi(x)=C+\alpha d^3(x)\in C^2(\overline{\Omega_{1,\delta}}).
$$
We claim that for $C$ and $\alpha$ sufficiently large, then $\psi$ satisfies
\begin{equation*}
 \left\{\begin{array}{cl}
 \PucF(x,D^2\psi)+\mu\psi\leq k(x)\psi^p  & \mbox{in }\ \Omega_{1,\delta}\\
\psi\geq\widetilde{\phi_{\gO_2}} & \mbox{on }\ \partial\Omega_{1,\delta}\cap\Omega_1.\\
\end{array}
\right.
 \end{equation*}
For any $x\in\partial\Omega_{1,\delta}\cap\Omega_1$, it holds $d(x)=\delta$ and 
\begin{equation}\label{7geneq4'}
\psi(x)=C+\alpha\delta^3=\max_{\left\{d(x)=\delta\right\}}\widetilde{\phi_{\gO_2}}(x)
\end{equation} 
by selecting 
\begin{equation}\label{7geneq4}
\alpha\delta^3=\left(\frac{\displaystyle\max_{\left\{d(x)=\delta\right\}}\phi_{\gO_2}}{\displaystyle\min_{\overline{\gO_1}}\phi_{\gO_2}}-1\right)C.
\end{equation} 
Moreover 
\begin{equation}\label{7geneq4''}
\psi(x)=C\qquad\forall x\in\partial \Omega_1.
\end{equation} 
For $x\in\Omega_{1,\delta}$, the eigenvalues $e_i(D^2\psi(x))$ of
$$
D^2\psi(x)=6\alpha d(x) Dd(x)\otimes Dd(x)+3\alpha d^2(x)D^2d(x)
$$
are $$e_N(D^2\psi(x))=6 \alpha d(x)$$ with $Dd(x)$ as corresponding eigenvector, and $$e_i(D^2\psi(x))=3\alpha d^2(x)e_i(D^2d(x)),$$  $e_i(D^2d(x))$ being,  for $i=1,\ldots, N-1$,  the eigenvalues of $D^2d(x)$ corresponding to eigenvectors which are orthogonal to $Dd(x)$. Since $d\in C^2(\overline{\Omega_{1,\delta}})$, we can assume, reducing $\delta$ if necessary, that 
$$
e_i(D^2\psi(x))\leq 6\alpha\delta\qquad\forall x\in \Omega_{1,\delta}.
$$
Hence,  from the uniform ellipticity assumptions\eqref{hypop}-\eqref{hypop2}, using \eqref{7geneq4} and the facts that $\displaystyle \inf_{\Omega_{1,\delta}}k>0$ and $p>1$, we infer that for any $x\in\Omega_{1,\delta}$ 
\begin{equation}\label{7geneq5}
\begin{split}
\PucF(x,D^2\psi(x))+\mu\psi(x) &\leq
\mathcal{M}^+_{\lambda, \Lambda}(D^2\psi(x))+\mu\psi(x)\\&\leq 6\alpha\delta\Lambda N+\mu(C+\alpha\delta^3)\\
&=\left(\frac{6\Lambda N}{\delta^2}\left(\frac{\displaystyle\max_{\left\{d(x)=\delta\right\}}\phi_{\gO_2}}{\displaystyle\min_{\overline{\gO_1}}\phi_{\gO_2}}-1\right)+\frac{\displaystyle\max_{\left\{d(x)=\delta\right\}}\phi_{\gO_2}}{\displaystyle\min_{\overline{\gO_1}}\phi_{\gO_2}}\,\mu\right)C\\
&\leq \left(\inf_{\Omega_{1,\delta}}k\right) C^p\leq k(x)\psi^p(x)
\end{split}
\end{equation}
provided $C=C(\delta,\Lambda,\mu, N,\Omega_1,k,\phi_{\Omega_2})$ is sufficiently large. For later purpose we can further suppose that  
\begin{equation}\label{7geneq6}
C^{p-1}\geq\frac{\mu}{\displaystyle\inf_{\Omega\backslash{\Omega_1}}k}.
\end{equation}
We claim that the function
$$
w^+(x)=\left\{\begin{array}{cl}
C &  \text{if $x\in\Omega\backslash\Omega_1$}\\
\min\left\{\widetilde{\phi_{\gO_2}}(x),\psi(x)\right\} &  \text{if $x\in \Omega_{1,\delta}$}\\
\widetilde{\phi_{\gO_2}}(x) & \text{if $x\in\Omega_1\backslash\Omega_{1,\delta}$}
\end{array}\right.
$$
is the desired supersolution. For this, we first note that $w^+\in C(\overline\Omega)$ in view of \eqref{7geneq2}-\eqref{7geneq4'}-\eqref{7geneq4''} and it is positive in $\Omega$. \\Now, let $x_0\in\Omega$ and  let $\varphi\in C^2(\Omega)$ be a test function touching $w^+$ from below at $x_0$. \\
Suppose first that $x_0\in\Omega_1$. If $w^+(x_0)=\widetilde{\phi_{\gO_2}}(x_0)$, then $\varphi$ touches $\widetilde{\phi_{\gO_2}}$ by below at $x_0$, and therefore by \eqref{7geneq3} we have
$$
\PucF(x_0,D^2\varphi(x_0))+\mu\varphi(x_0)\leq0\leq k(x_0)\varphi^p(x_0).
$$
If instead $w^+(x_0)=\psi(x_0)<\widetilde{\phi_{\gO_2}}(x_0)$, then $x_0\in \Omega_{1,\delta}$ because of \eqref{7geneq4'}. Hence $\varphi$ touches $\psi$ by below at $x_0$ and, in view of \eqref{7geneq5}, we have
$$
\PucF(x_0,D^2\varphi(x_0))+\mu\varphi(x_0)\leq k(x_0)\varphi^p(x_0).
$$
Now we suppose that $x_0\in \Omega\backslash \Omega_1$. The conclusion in the case $x_0\in\Omega\backslash\overline{\Omega_1}$ simply follows by the inequality \eqref{7geneq6}. Lastly, if $x_0\in \partial\Omega_1$, then for any $x\in\Omega_{1,\delta}\cup \Omega\backslash \Omega_1$
$$
\varphi(x)\leq \zeta(x)=\left\{\begin{array}{cl}
C & \text{in $\Omega\backslash \Omega_1$}\\
\psi(x) & \text{in $\Omega_{1,\delta}$.}
\end{array}\right.\quad\;\text{and}\quad\;\varphi(x_0)=\zeta(x_0).
$$
Since $\zeta$ is a $C^2$ function in a neighbourhood of $x_0$ (because of \eqref{7geneq4''}), then using \eqref{7geneq6} we conclude
$$
\PucF(x_0,D^2\varphi(x_0))+\mu\varphi(x_0)\leq \mu C\leq k(x_0)\varphi^p(x_0).
$$

\end{proof}
We are now in a position to prove the existence part of Theorem \ref{KW}.

\begin{proposition}\label{exiKW} 
Suppose that $\PucF$ satisfies \eqref{hypop}, \eqref{hypop2} and \eqref{hypop2'}.
\begin{enumerate}
\item If $k\in C(\Omega)$  satisfies \eqref{hypop1k} and $\mu>\lambda_1^+(\Omega) $, then there exists a unique positive solution of \eqref{eq1}.
\item If $k\in C(\Omega)$ satisfies \eqref{hypop2k} and $\lambda_1^+(\Omega)<\mu< \lambda_1^+(\Omega_0)$, then there exists a unique positive solution of \eqref{eq1}.
\end{enumerate}
\end{proposition} 
\begin{proof}
We begin by constructing a subsolution for \eqref{eq1}. Recall that $ k(x)\leq k_1$ and let 
$$0<\beta\leq\left(\frac{\mu-\lambda_1^+(\Omega)}{k_1}\right)^{\frac{1}{p-1}}.$$
With this choice of $\beta$, the function $w^-(x)=\beta\phi_\gO(x)$ is a subsolution of \eqref{eq1}:
$$\PucF(x,D^2w^-)+\mu w^-=(\mu-\lambda_1^+(\Omega))w^-\geq k(x)(w^-)^p \ \mbox{in}\ \Omega,\;\, w^-=0\  \mbox{on}\ \partial\Omega.$$
Remark $w^-$ is a subsolution in both cases \eqref{hypop1k} and \eqref{hypop2k}.

We now proceed with the case where $k$ satisfies hypothesis \eqref{hypop1k}, so that $k_0\leq k(x)\leq k_1$. In this case we can chose as a supersolution the constant function $w^+(x)=C=\left(\frac{\mu}{k_0}\right)^{\frac{1}{p-1}}$, since
$$\PucF(x,D^2w^+)+\mu w^+=\mu C= k_0C^{p}\leq k(x)(w^+)^p.$$

When $k$ satisfies hypothesis \eqref{hypop2k}, the construction of the supersolution is more involved and it has been given in Lemma \ref{supso}.

\vspace{0.5cm}

Let $C_1=\max_{\overline \Omega}(w^+)$ and we choose  $\mu_0>\mu$ such that $\mu_0>pk_1C_1^{p-1}$. This implies that the function
$$f(x,t)=k(x)t^p-\mu_0t$$
is decreasing in $t\in [0,C_1]$ for any $x\in\Omega$. Without loss of generality we can always suppose that $\beta<\min_{\overline \Omega}(w^+)$.

Now, under either of the hypothesis on $k$, we can construct a monotone decreasing sequence of positive functions.

Let $u_0=w^+$, and $u_{n+1}$ be the solution of the following Dirichlet problem:
$$\left\{
\begin{array}{cl}
\PucF(x,D^2u_{n+1})+(\mu-\mu_0) u_{n+1}= k(x)u_n^p-\mu_0u_n & \text{in $\Omega$}\\
u_{n+1}=0 & \text{on $\partial\Omega$.}
\end{array}\right.
$$
Observe that
$$ \PucF(x,D^2u_{1})+(\mu-\mu_0) u_{1}=f(x,w^+)\geq \PucF(x,D^2w^+)+(\mu-\mu_0) w^+\quad\text{in $\Omega$}$$

Furthermore
$$ \PucF(x,D^2u_{1})+(\mu-\mu_0) u_{1}=f(x,w^+)\leq f(x,w^-)\leq  \PucF(x,D^2w^-)+(\mu-\mu_0) w^-\quad\text{in $\Omega$}.$$
Since the operator is proper, from the boundary condition we get
 $$w^-\leq u_1\leq u_0\leq C_1.$$
Suppose, for induction sake that, for some $n\in\mathbb N$, we have $w^-<u_{n}\leq u_{n-1}\leq w^+$. Then
$$ \PucF(x,D^2u_{n+1})+(\mu-\mu_0) u_{n+1}=f(x,u_n(x))\geq f(x,u_{n-1}(x))=\PucF(x,D^2u_{n})+(\mu-\mu_0) u_{n}\quad\text{in $\Omega$},$$
and
$$ \PucF(x,D^2u_{n+1})+(\mu-\mu_0) u_{n+1}=f(x,u_n(x))\leq f(x,w^-)\leq\PucF(x,D^2w^-)+(\mu-\mu_0) w^-\quad\text{in $\Omega$}.$$
Again, the comparison principle in  $\gO$ implies that
$$w^-\leq u_{n+1}\leq u_n\leq w^+.$$

Since the sequence $\left\{u_n\right\}$ is bounded, by standard $C^{0,\alpha}_\text{loc}(\Omega)$ estimates, see e.g. \cite[Proposition 4.10]{CC}, the functions $u_n$ are also locally equi-H\"older continuous. Since $w^-=u_1=0$ on $\partial \Omega$, then 
the monotone bounded sequence $u_n$ converges to $\bar u\geq \alpha\phi_\gO>0$ solution of
\eqref{eq1}.
\end{proof}

\begin{proof}[Proof of Theorem \ref{KW}]
It follows from Propositions \ref{23marprop1} and \ref{exiKW}.
\end{proof}

In the following propositions, we establish the existence  and uniqueness of positive solutions to the Dirichlet problem \eqref{eq1} posed in $\Omega\backslash\Omega_0$, with general continuous boundary condition on $\partial\Omega_0$.
\begin{proposition}\label{exi} Let $k\in C(\Omega\setminus\overline\Omega_0)$ and let $\phi\in C(\partial\Omega_0)$ be such that $k>0$ in $\Omega\setminus\overline\Omega_0$ and $\phi\geq 0$ in $\partial\Omega_0$ not identically zero. Then, for any $\mu\in\R$,
 there exists a unique  positive solution $u\in C(\overline\Omega)$ of
\begin{equation}\label{eq3}
 \left\{\begin{array}{cl}
 \PucF(x,D^2u)+\mu u=k(x)u^p & \mbox{in }\ \Omega\setminus\overline\Omega_0\\
u=0 & \mbox{on }\ \partial\Omega\\
u=\phi & \mbox{on }\ \partial\Omega_0
\end{array}
\right.
 \end{equation}
\end{proposition}
\begin{proof} 

{\bf Step 1.} For $\delta>0$, we shall prove that there exists $u_\delta$ positive solution of
\begin{equation}\label{eq3d}
 \left\{\begin{array}{cl}
 \PucF(x,D^2u)+\mu u=(k(x)+\delta)u^p & \mbox{in }\ \Omega\setminus\overline\Omega_0\\
u=0 & \mbox{on }\ \partial\Omega\\
u=\phi & \mbox{on }\ \partial\Omega_0.
\end{array}
\right.
 \end{equation}
We begin by constructing a super-solution. Let $\mu^{\star}>\max \{\mu,\lambda_1^+(\Omega)\} $ 

We are in the hypothesis of (1) in Proposition \ref{exiKW}, hence there exists $w$ solution of
$$
 \left\{\begin{array}{cl}
 \PucF(x,D^2w)+\mu^\star w=(k(x)+\delta)w^p & \mbox{in }\ \Omega\\
w=0 & \mbox{on }\ \partial\Omega.
\end{array}
\right.
$$
Now take $M>1$ such that $Mw^+\geq\phi$ on $\partial\Omega_0$. It is easy to see that $w^+=Mw$ is a supersolution of \eqref{eq3d}.
On the other hand $w^-\equiv0$ is a subsolution. 

We now proceed to construct a sequence such that $u_0=w^+$ and recursively $u_{n+1}$ is the solution of
$$
 \left\{\begin{array}{cl}
\PucF(x,D^2u_{n+1})+(\mu-\mu_0) u_{n+1}=(k(x)+\delta)u_n^p -\mu_0u_n & \mbox{in }\ \Omega\\
u_{n+1}=0 & \mbox{on }\ \partial\Omega\\
u_{n+1}=\phi & \mbox{on}\ \partial\Omega_0,
\end{array}
\right.
$$
with $\mu_0$ chosen large enough that $f(x,t)=(k(x)+\delta)t^p-\mu_0 t$ is decreasing in $(0, \max_{\overline\Omega} w^+)$ and $\mu_0>\mu$.

Proceeding as in the proof of Proposition \ref{exiKW}, it is easy to see that the sequence is monotone decreasing and that uniform boundary estimates of Corollary \ref{bdryest} hold. 
Hence the sequence $u_n$ converges to the desired solution. We leave the details to the interested reader.

\smallskip
\noindent{\bf Step 2.} We want to prove that $u_\delta$ converges, as $\delta\to0$, to $u_0$ a solution of \eqref{eq3}. For this, we first note that, for any $\delta>0$, we have that
$$
\PucF(x,D^2u_\delta)+\mu u_\delta={(k(x)+\delta)}u_\delta^p\geq k(x)u_\delta^p\quad\text{in $\Omega$}.
$$
Hence, using Corollary \ref{bdryest} in $\Omega\backslash\overline\Omega_0$, we have that
\begin{equation}\label{23mareq2}
u_\delta(x)\leq\psi_\phi(x)+Cd(x),
\end{equation}
where $\psi_\phi$ is the solution of 
$$
\PucF(x,D^2\psi_\phi)=0\;\;\;\text{in $\Omega\setminus\overline\Omega_0$}\,,\;\;\psi_\phi=0\;\;\;\text{on $\partial\Omega$}\,\;\;\psi_\phi=\phi\;\;\;\text{on $\partial\Omega_0$}.
$$
Moreover, $u_\delta$ is monotone with respect to $\delta$. Indeed, if $\delta_1<\delta_2$, then
$$
\PucF(x,D^2u_{\delta_1})+\mu u_{\delta_1}\leq(k(x)+\delta_2)u^p_{\delta_1}\quad\,\text{in $\Omega\backslash\overline\Omega_0$}. 
$$
Combining Propositions \ref{strongcomp}-\ref{comp} we obtain $u_{\delta_2}<u_{\delta_1}$ in $\Omega\backslash\overline\Omega_0$. 
So, if we fix $\bar\delta>0$, using \eqref{23mareq2}, we deduce that for any $\delta\leq\bar\delta$ it holds 
%
 $$ u_{\bar\delta}\leq u_\delta\leq \psi_\phi+Cd\quad\,\text{in $\overline\Omega\backslash\Omega_0$}.$$
So $(u_\delta)_\delta$ is equibounded and, by $C^{0,\alpha}_\text{loc}$ estimates,  it converges to $u_0$ solution of \eqref{eq3}. Uniqueness of the positive solution to \eqref{eq3} follows from Proposition \ref{comp}. 
\end{proof}

We can now prove the existence of blow-up solutions.
\begin{theorem}\label{bup}  Let $k\in C(\Omega\backslash\overline\Omega_0)$ be such that $\displaystyle\inf_{\Omega\backslash\Omega_1}k>0$ whenever $\Omega_1\subsetneqq\Omega$ such that $\Omega_0\subset\subset\Omega_1$. Then, for any $\mu\in\R$,  there exist   $\overline U_\mu\geq \underline U_\mu>0$   respectively the smallest and the largest solution of 
\begin{equation}\label{eq4}
 \left\{\begin{array}{cl}
 \PucF(x,D^2u)+\mu u=k(x)u^p & \mbox{in }\ \Omega\setminus\overline\Omega_0\\
u=0 & \mbox{on }\ \partial\Omega\\
u=+\infty & \mbox{on }\ \partial\Omega_0,
\end{array}
\right.
 \end{equation}
i.e. such that any solution $u$ of \eqref{eq4} satisfies
$$\underline U_\mu\leq u\leq\overline U_\mu. $$
\end{theorem}
\begin{proof} 
 Let $\delta>0$ be sufficiently small that there exists (see the proof of Theorem 2.4 in  \cite{DuHu}) $ k^\star\in C^2(\overline\Omega_{0,\delta})$ such that, for some function $f$,
$$k^\star(x):=f(d_{\Omega_0}(x))\leq k(x). $$
Let $\gamma>\frac{3}{p-1}$ so that $w(x)=k^\star(x)^{-\gamma}$ satisfies, for some $C_0$ depending only on the ellipticity condition of $\PucF$, $\mu$ and the $C^2$ norm of $k^\star$, the inequality
 $$\PucF(x,D^2w)+\mu w\leq  C_0k^\star(x)^{-\gamma-2}\quad\;\forall x\in \Omega_{0,\delta}.$$
Reducing $\delta$ if necessary, we may further assume that $\Omega_\delta\cap\Omega_{0,\delta}=\emptyset$ and that 
$C_0k^\star(x)^{-\gamma-2}\leq k(x)w^p(x)$ for any $x\in\Omega_{0,\delta}$. Hence 
\begin{equation}\label{5febeq1}
\PucF(x,D^2w)+\mu w\leq k(x)w^p(x)\quad\;\forall x\in \Omega_{0,\delta}.
\end{equation}
Let $u_n$ be the sequence of solutions of  \eqref{eq3} with $\phi(x)=n$. Pick $\varepsilon$ such that 
\begin{equation}\label{5febeq2}
0<\varepsilon<\min\left\{\min_{\left\{d_{\Omega}(x)=\delta\right\}}u_1\,,\,\min_{\left\{d_{\Omega_0}(x)\leq\delta\right\}}w
\right\}
\end{equation}
and define
$$
\psi(x)=\left\{\begin{array}{cl}
\varepsilon & \text{if $d_{\Omega_0}(x)\geq\delta$}\\
\varepsilon +\alpha(\delta-d_{\Omega_0}(x))^3 & \text{if $\frac\delta2\leq d_{\Omega_0}(x)<\delta$,}
\end{array}\right.
$$
where $$\alpha=\frac{\max_{\left\{d_{\Omega_0}(x)=\frac\delta2\right\}}w-\varepsilon}{{(\delta/2)^3}}$$ is chosen in such a way 
\begin{equation}\label{5febeq3}
\psi(x)\geq w(x)\quad\;\text{if  $d_{\Omega_0}(x)=\frac\delta2$.}
\end{equation}
 Note that $\psi\in C^2(\Omega\backslash\Omega_{0,\frac\delta2})$ and that 
\begin{equation}\label{5febeq4}
\PucF(x,D^2\psi)+\mu \psi\leq C\quad\;\forall x\in\Omega\backslash\Omega_{0,\frac\delta2}
\end{equation}
for some constant $C=C(\varepsilon,\alpha, \delta,\mu)$. \\
Let 
$$
W(x)=\left\{\begin{array}{cl}
w(x) & \text{if $d_{\Omega_0}(x)<\frac\delta2$}\\
\min\left\{w(x),\psi(x)\right\} & \text{if $\frac\delta2\leq d_{\Omega_0}(x)<\delta$}\\
\psi(x) & \text{if $d_{\Omega_0}(x)\geq\delta$ and $d_{\Omega_0}(x)\geq\delta$}\\
\min\left\{\psi(x),u_1(x)\right\} & \text{if $d_{\Omega_0}(x)<\delta$}.
\end{array}\right.
$$
In view of \eqref{5febeq2}-\eqref{5febeq3} the function $W$ is continuous in $\Omega\backslash\overline{\Omega_0}$. We claim that for 
$$
M^{p-1}\geq\max\left\{1\,,\,\frac{C}{\varepsilon ^p\displaystyle\inf_{\Omega\backslash\Omega_{0,\frac\delta2}} k}\right\},
$$
the function $W^+(x)=MW(x)$ is a supersolution of 
 $$\PucF(x,D^2u)+\mu u= k(x)u^p\quad\;\text{in $\Omega\backslash\overline{\Omega_0}$.}$$
By the choice of $M$, it is easy to realize that the functions $Mw$, $M\psi$ and $M u_1$ are supersolutions of the above equation respectively in $\Omega_{0,\delta}$, $\Omega\backslash\overline{\Omega_{0,\frac\delta2}}$ and $\Omega\backslash{\overline\Omega_0}$. 
Hence  for $x\notin\partial\Omega_{0,\frac\delta2}\cup \partial\Omega_{0,\delta}\cup \partial\Omega_{\delta}$, the function $W^+$ satisfies, in the viscosity sense, 
the inequality $\PucF(x,D^2u(x))+\mu u(x)\leq k(x)u^p(x)$. 
To check the supersolution property in the remaining cases, let $\varphi\in C^1(\Omega\backslash\overline{\Omega_0})$ touching from below $W^+$ at $x_0\in\partial\Omega_{0,\frac\delta2}\cup \partial\Omega_{0,\delta}\cup \partial\Omega_{\delta}$. If $x_0\in\partial\Omega_{0,\frac\delta2}$, from \eqref{5febeq3} we infer that $\varphi$ is in fact a test function for $Mw$ and the required inequality holds. 
If instead $x_0\in\partial\Omega_{0,\delta}\cup \partial\Omega_{\delta}$, then  $W^+\equiv M\psi$ in a neighborhood on $x_0$ because of \eqref{5febeq2}, and the claim is proved.  

\smallskip

Hence by Proposition \ref{comp} the sequence $u_n$ is monotone increasing and bounded above by $W^+$ and it converges to $\underline U_\mu$ solution of \eqref{eq4}. Furthermore it is easy to see that $\underline U_\mu$ is the smallest of all blow up solutions of \eqref{eq4}.

In order to construct the largest of blow up solutions, consider a sequence of domains $\Omega_n$ such that for every $n\in\N$, $\Omega_{n+1}\subset\Omega_n\subset\Omega$ 
and $\Omega_0=\cap_n\Omega_n$. Let $v_n$ be a solution of 
$$
 \left\{\begin{array}{cl}
 \PucF(x,D^2v_n)+\mu v_n=k(x)v_n^p & \mbox{in }\ \Omega\setminus\overline\Omega_n\\
v_n=0 & \mbox{on }\ \partial\Omega\\
v_n=+\infty & \mbox{on }\ \partial\Omega_n.
\end{array}
\right.
$$
Then $v_n$ is a monotone decreasing sequence of solutions converging to $\overline U_\mu$ a solution of \eqref{eq4}. It is easy to see that it is the largest of all solutions of \eqref{eq4}.
\end{proof} 


\section{Asymptotic results}

\begin{proof}[Proof of Theorem \ref{converg1}]
By the monotonicity proved in Corollary \ref{mon},  it is enough to study the asymptotic behaviour of the sequence $\{u_n=u_{\mu_n}\}$, where  $\{\mu_n\}$ is a monotone decreasing sequence converging to  $\lambda_1^+(\Omega)$. The sequence $\{u_n\}$ is monotone decreasing and uniformly bounded, hence it converges locally uniformly to 
 $u_0\geq 0$ nonnegative solution of
$$
 \PucF(x,D^2u_0)+\lambda_1^+(\Omega) u_0=k(x)u_0^p\geq 0\;\;\text{in $\Omega$,}\;u_0=0\;\;\text{on $\partial\Omega$}. $$
By Proposition \ref{maxp}, $u_0\equiv\tau\phi_1$ for some $\tau\geq0$. Since  $k\not \equiv 0$, then necessarily $\tau=0$ and $u_0\equiv0$. 

The above argument shows  that $u_\mu\to 0$ uniformly in $\Omega$ as $\mu\to \lambda_1^+(\Omega)$. Observing further that 
 the functions $v_\mu=\frac{u_\mu}{\| u_\mu\|_\infty}$ satisfy
 $$
 \left\{\begin{array}{cl}
 \PucF(x,D^2v_\mu)+\mu v_\mu= \|u_\mu\|_\infty^{p-1} k(x) v_\mu^p & \mbox{in }\ \Omega\\
 v_\mu>0 
 & \mbox{in }\ \Omega \\
v_\mu=0 & \mbox{on }\ \partial\Omega,
\end{array}
\right.
$$
standard $C^{0,\alpha}_{\text{loc}}$ elliptic estimates imply that, locally uniformly in $\Omega$,  $v_\mu\to v$, where $v$ is a nonnegative  solution of 
$$
\left\{\begin{array}{cl}
 \PucF(x,D^2v)+\lambda_1^+(\Omega) v= 0 & \mbox{in }\ \Omega\\
v=0 & \mbox{on }\ \partial\Omega.
\end{array}
\right.
$$
Moreover, denoting by $x_\mu$ any maximum point of $v_\mu$,  by Corollary \ref{bdryest} we also infer that there exists a positive constant $C$, independent of $\mu$ (since $\mu$ is bounded in this argument), such that $1=v_\mu(x_\mu)\leq C d(x_\mu)$. Thus $(x_\mu)_\mu$ is bounded away from the boundary of $\Omega$ and the uniform local convergence of $(u_\mu)_\mu$ ensures that $\left\|v\right\|=1$. By Proposition \ref{maxp}, we conclude that $v=\phi_\Omega$.
\end{proof}

In order to provide a proof for Theorem 3, we need to assume the H\"older continuity of the coefficients appearing in equation \eqref{eq1}. Precisely, we assume that there exist $\alpha\in (0,1)$ and $C>0$ such that
\begin{equation}\label{H}\tag{H}
k\in C^{0,\alpha}(\Omega) \quad \hbox{ and } \quad |\PucF(x, M)-\PucF(y,M)|\leq C\, \|M\| |x-y|^\alpha\qquad \forall\, x\, , y\in \Omega\, ,\ M\in \mathcal{S}_N\, .
\end{equation}
Under this additional assumption, solutions of equation \eqref{eq1} satisfying smooth boundary conditions enjoy global $C^{1,\alpha}$ a priori estimates, see \cite{SS}.

{\begin{proof}[Proof of Theorem \ref{converg2}]

We partially follow the proof  of Du and Huang \cite{DuHu}. For $n\geq 1$, let us introduce the set
$$
\Omega_n=\left\{ x\in\Omega\ : d_{\Omega_0}(x)<\frac{1}{n}\right\}
$$
which will be assumed, without loss of generality, to be contained in $\Omega$. We further set $\mu_n=\lambda_1^+(\Omega_n)$. Then, $\lambda_1^+(\Omega) < \mu_n< \lambda_1^+(\Omega_0)$ and $\mu_n\to \lambda_1^+(\Omega_0)$. As in the proof of Theorem \ref{converg1}, in order to study the asymptotic behavior of $u_\mu$ as $\mu\to \lambda_1^+(\Omega_0)$, it is enough to analyze the monotone increasing sequence $u_n=u_{\mu_n}$.

{\bf Step 1}: For any compact subset $K\subset \Omega_0$, one has
$$
u_n \to +\infty \quad \hbox{ uniformly  in } K\, .
$$
Let us set
$$
\alpha\, := \inf_{\Omega_0}u_1\, ,\qquad \beta \, := \min_{K} \phi_{\Omega_0}\, .
$$
Then, both $\alpha$ and $\beta$ are positive, and $u_n\geq \alpha$ in $\Omega_0$ for all $n\geq 1$. For $M>0$ fixed arbitrarily large, we can select a compact set $K^*$, with $K\subset K^*\subset \Omega_0$ such that
$$
\phi_{\Omega_0} \leq \frac{\alpha \beta}{2M}\qquad \hbox{ on } \partial K^*\, .
$$
Let us further denote by $\phi_n$ the eigenfunction $\phi_{\Omega_n}$. We observe that $\phi_n\to \phi_{\Omega_0}$ uniformly in $K^*$. Then, for $n$ sufficiently large, we have
$$
\phi_n < \frac{\alpha \beta}{M}\quad \hbox{ on } \partial K^*\, , \quad \phi_n> \frac{\beta}{2}\quad \hbox{ in } K\, .
$$
Moreover, both $u_n$ and $\frac{M}{\beta} \phi_n$ are solutions of the equation
$$
\PucF(x,D^2v)+\mu_n v= 0 \quad  \mbox{in  int}(K^*)
$$
and $\frac{M}{\beta} \phi_n \leq \alpha \leq u_n$ on $\partial K^*$. Since $\mu_n< \lambda_1^+(\Omega_0)< \lambda^+_1({\rm int}(K^*))$,  by Corollary \ref{compmu}  we deduce that $u_n\geq \frac{M}{\beta} \phi_n$ in $K^*$. Hence, $u_n\geq \frac{M}{2}$ in $K$ and, by the arbitrariness of $M>0$, we obtain the conclusion of Step 1.
\medskip

{\bf Step 2}: $\inf_{\Omega_0} u_n \to +\infty$.

By the minimum principle, for all $n\geq 1$ there exists   a minimum point $x_n\in \partial \Omega_0$ for $u_n$ in $\overline{\Omega_0}$.  Let us assume, by contradiction, that the conclusion of Step 2 is false. Then, up to a subsequence still denoted with $n$, we have that $\{ u_n(x_n)\}$ is a bounded sequence.

By the uniform interior ball property for $\Omega_0$, there exists $R>0$ and, for each $n\geq 1$, a point $y_n\in \Omega_0$ such that $\overline{B_R(y_n)} \subset \overline{\Omega_0}$ and $\overline{B_R(y_n)}\cap \partial \Omega_0= \{x_n\}$. 

For $\sigma\, ,\  c_n>0$ to be suitably fixed, let us consider the function
$$
v_n(x)= u_n(x_n)+ c_n\left( e^{-\sigma |x-y_n|^2}- e^{-\sigma R^2}\right)\, .
$$
A direct computation shows that,  for $\sigma>0$ sufficiently large independently of $c_n$, $v_n$ satisfies
$$
\PucF(x, D^2v_n)+\mu_n v_n\geq 0\quad \hbox{ in } B_R(y_n)\setminus B_{R/2}(y_n)\, .
$$
Moreover, on $\partial B_R(y_n)$ we have $v_n(x)=u_n(x_n)\leq u_n(x)$,  as well as $v_n(x)\leq u_n(x)$ for $x\in \partial B_{R/2}(y_n)$ by choosing
$$
c_n= \frac{\inf_{B_{R/2}(y_n)}u_n-u_n(x_n)}{e^{-\sigma R^2/4}-e^{-\sigma R^2}}\, .
$$
We note that, by Step 2 and the assumption on the boundedness of $\{u_n(x_n)\}$, we have that $c_n\to +\infty$.

Corollary \ref{compmu} implies that $u_n\geq v_n$ in $B_R(y_n)\setminus B_{R/2}(y_n)$, and $u_n(x_n)=v_n(x_n)$. Hence, setting $\nu_n=\frac{y_n-x_n}{|y_n-x_n|}$, we obtain
\begin{equation}\label{unbound}
\frac{\partial u_n}{\partial \nu_n}(x_n)\geq \frac{\partial v_n}{\partial \nu_n}(x_n)= c_n \left( 2 \sigma R e^{-\sigma R^2}\right) \to +\infty\, .
\end{equation}
On the other hand, we can consider the solution $w_n$ of the problem
$$
\left\{\begin{array}{lc}
 \PucF(x,D^2w_n)+\mu_n w_n=k(x)w_n^p & \mbox{in }\ \Omega\setminus\overline\Omega_0\\
 w_n>0 
 & \mbox{in }\ \Omega \setminus\overline\Omega_0\\
w_n=0 & \mbox{on }\ \partial\Omega\\
w_n=u_n(x_n) & \mbox{on }\ \partial\Omega_0
\end{array}
\right.
$$
By the comparison principle given in Proposition \ref{comp}, we deduce that $u_n\geq w_n$ in $\Omega \setminus \Omega_0$ and, obviously, $u_n(x_n)=w_n(x_n)$. Hence, we also have
$$
\frac{\partial u_n}{\partial \nu_n}(x_n)\leq \frac{\partial w_n}{\partial \nu_n}(x_n)\, .
$$
Since $\{u_n(x_n)\}$ is bounded, the maximum principle implies that $\{w_n\}$ is uniformly bounded in $\overline{\Omega}\setminus \Omega_0$. By elliptic $C^1$ estimates, see \cite{SS}, we then obtain that $\{w_n\}$ is bounded in $C^1( \overline{\Omega}\setminus \Omega_0)$, hence there exists $C_0>0$ such that
$$
\frac{\partial u_n}{\partial \nu_n}(x_n)\leq \frac{\partial w_n}{\partial \nu_n}(x_n)\leq C_0\, ,
$$
which is a contradiction to \eqref{unbound}.
\medskip

{\bf Step 3}: $u_n\to \underline{U}$ locally uniformly in $\Omega\setminus \Omega_0$.

By comparison, we have that $u_n\leq \underline{U}$ in $\Omega\setminus \Omega_0$. Hence, the monotone increasing sequence $\{u_n\}$ is locally uniformly bounded in $\Omega\setminus \overline{\Omega_0}$. Thus, by using also Step 2, we obtain that $u_n\to V$  locally uniformly in $\Omega\setminus \overline{\Omega_0}$, where $V$ is a solution of problem \eqref{eq2} satisfying further $V\leq \underline{U}$. By minimality of $\underline{U}$, we then conclude $V\equiv \underline{U}$.
\medskip

{\bf Step 4}: $\frac{u_\mu}{\|u_\mu\|_{L^\infty (\Omega_0)}}\rightarrow \phi_{\Omega_0}\ \hbox{uniformly in}\ \overline{\Omega_0}$.

Let us start by considering  the sequence of functions $v_n\, :=\frac{u_n}{\|u_n\|_{\infty}}=\frac{u_n}{\|u_n\|_{L^\infty (\Omega)}}$, where $u_n=u_{\mu_n}$ and $\{ \mu_n\}$ is any sequence belonging to $\left( \lambda_1^+(\Omega)\, , \lambda_1^+(\Omega_0)\right)$ and converging to $\lambda_1^+(\Omega_0)$.

Since $\{  v_n\}$ is bounded in $L^\infty(\Omega)$, there exists a limit function $v\in L^\infty(\Omega)$, with $\|v\|_{L^\infty(\Omega)}\leq 1$, such that, up to a subsequence still denoted by $\{v_n\}$, one has $v_n \stackrel{\ast}{\rightharpoonup} v$ in $L^\infty(\Omega)$, that is
$$
\int_\Omega v_n\, \psi\, dx \to \int_\Omega v\, \psi\, dx \quad \hbox{ for all } \psi\in L^1(\Omega)\,.
$$
On the other hand, by Steps 2 and 3 above,  we know that $\{u_n\}$ is locally uniformly bounded in $\Omega\setminus \overline{\Omega_0}$ and $\| u_n\|_{\infty}\to +\infty$, so that   $v=0$ in $\Omega\setminus \overline{\Omega_0}$ and $v_n\to 0$ locally uniformly in $\Omega\setminus \overline{\Omega_0}$. Furthermore, the functions $v_n$ satisfy 
\begin{equation}\label{vn}
\PucF(x,D^2v_n)+\mu_n  v_n=k(x)\, \|u_n\|_{\infty}^{p-1} v_n \left\{
\begin{array}{ll} \geq 0 & \hbox{ in } \Omega\\[1ex] 
=0 & \hbox{ in }\Omega_0
\end{array}\right.
\end{equation}
and they are uniformly bounded. By applying standard interior elliptic estimates, we deduce that, again up to a subsequence neglected in the notation, $\{v_n\}$ is locally uniformly converging to a continuous solution of the eigenvalue equation in $\Omega_0$. Hence, the limit function $v$ is continuous in $\Omega_0$ and satisfies 
\begin{equation}\label{limitv}
\PucF(x,D^2v)+\lambda_1^+(\Omega_0) v=0\, ,\quad v\geq 0 \quad  \hbox{in }\Omega_0\, .
\end{equation}
By the strong minimum principle, we have that either $v>0$ in $\Omega_0$ or $v\equiv 0$ in $\Omega_0$. Let us prove  that the latter cannot occur. Arguing by contradiction, let us assume that $v\equiv 0$ in $\Omega_0$, hence $v=0$ almost everywhere in $\Omega$. Since $v_n$ satisfies the first inequality in \eqref{vn}, we can apply the so called local maximum principle for subsolutions of uniformly elliptic equations, see   \cite[Theorem 4.8]{CC}, yielding in particular that for any ball $B_{2r}(x_0) \subset\subset \Omega$ one has
$$
\sup_{B_r(x_0)} v_n \leq C\, \left[ \frac{1}{|B_{2r}(x_0)|}\int_{B_{2r}(x_0)} v_n\, dx + r \| \mu_n v_n\|_{L^N(B_{2r}(x_0))} \right]\, ,
$$where $C>0$ is a universal constant depending only on $\lambda\, , \Lambda$ and $N$. Now, if $v_n \stackrel{\ast}{\rightharpoonup} 0=v$ in $L^\infty(\Omega)$, by using a covering argument, for any $\varepsilon >0$ sufficiently small we obtain the estimate
$$
\limsup_{n\to \infty} \sup_{\mathcal{N}_\epsilon} v_n\leq C\, \varepsilon\, ,
$$
where $\mathcal{N}_\varepsilon$ is the tubular neighborhood of $\partial \Omega_0$ with radius $\varepsilon$. This implies that $v_n\to 0$ uniformly in $\Omega$, which is a contradiction to $\|v_n\|_\infty = 1$.

Hence, we have proved that $v$ is a strictly positive solution of \eqref{limitv}. Now, since for every $t\geq0$ the function $z_t=t\phi_{\Omega_0}-v$ is a subsolution of $\mathcal {M}^+(D^2z_t)+\lambda_1^+(\Omega_0)z_t=0$ in $\Omega_0$ and $\limsup_{x\to\partial\Omega_0}z_t(x)\leq0$, 
 the argument in \cite[Theorem 4.1]{QS} implies that $v=\tau\, \phi_{\Omega_0}\chi_{\Omega_0}$ for some $\tau>0$, where $\chi_{\Omega_0}$ denotes the characteristic function of $\Omega_0$.  Finally,  by applying again the local maximum principle,  we can   uniformly estimate the functions  $v_n$ in small tubular neighborhoods of $\partial \Omega_0$ and obtain that $v_n\to v$ uniformly in $\Omega$. Since $\|v_n\|_\infty =1$, this implies that $\tau =1$ and $\|u_n\|_\infty =\| u_n\|_{L^\infty (\Omega_0)}$ for $n$ large enough. Hence, we obtain that, up to a subsequence, 
$$
\frac{u_n}{\|u_n\|_{L^\infty(\Omega_0)} }\to \phi_{\Omega_0}\quad \hbox{ uniformly in } \overline{\Omega_0}\, .
$$
The above proof can be applied to any subsequence extracted from $\left\{\frac{u_n}{\|u_n\|_{L^\infty(\Omega_0)}}\right\}$. Hence, we deduce the convergence of the whole sequence and, by the arbitrariness of $\{ \mu_n\}$, we get the statement of Step 4.
\end{proof}

{\bf Data Availability Statement} Data sharing is not applicable to this article as no datasets were generated or
analyzed in this article.
\end{document}